\documentclass[11pt]{amsart}
\usepackage{amsmath,amssymb,amsfonts}
\usepackage{tikz-cd}
\usepackage{amsbsy}
\usepackage{amscd}
\usetikzlibrary{matrix,arrows,decorations.pathmorphing}
\usepackage{graphicx}
\usepackage{float}
\usepackage{enumitem}
\usepackage{setspace}
\usepackage[margin=1.0in]{geometry}
\usepackage[all]{xy}

\newtheorem{thm}{Theorem}[section]
\newtheorem{lem}[thm]{Lemma}
\newtheorem{cor}[thm]{Corollary}
\newtheorem{prop}[thm]{Proposition}

\theoremstyle{definition}
\newtheorem{example}[thm]{Example}
\theoremstyle{definition}
\newtheorem{examples}[thm]{Examples}
\theoremstyle{definition}
\newtheorem{defn}[thm]{Definition}
\theoremstyle{definition}

\newcommand{\mc}[1]{\mathcal{#1}}
\newcommand{\e}[1]{\emph{#1}}

\newcommand{\la}{\langle}
\newcommand{\ra}{\rangle}
\newcommand{\tr}{\mathrm{tr}}

\newcommand{\rmv}[1]{}

\newcommand{\hs}{\hskip10pt}

\newcommand{\LO}{L^1(G)}
\newcommand{\LOQ}{L^1(\mathbb{G})}

\newcommand{\LOQH}{L^1(\widehat{\mathbb{G}})}

\newcommand{\LTQ}{L^2(\mathbb{G})}

\newcommand{\LI}{L^{\infty}(G)}
\newcommand{\LIQ}{L^{\infty}(\mathbb{G})}

\newcommand{\BH}{\mc{B}(H)}

\newcommand{\BLTQ}{\mc{B}(L^2(\mathbb{G}))}

\newcommand{\TCQ}{\mc{T}(L^2(\mathbb{G}))}

\newcommand{\vphi}{\varphi}

\newcommand{\Lphi}{\Lambda_\varphi}

\newcommand{\al}{\alpha}
\newcommand{\be}{\beta}
\newcommand{\lm}{\lambda}

\newcommand{\Lphis}{\Lambda_{\varphi\ten\varphi}}
\newcommand{\Gam}{\Gamma}
\newcommand{\om}{\omega}

\newcommand{\ten}{\otimes}
\newcommand{\oten}{\overline{\otimes}}
\newcommand{\pten}{\widehat{\otimes}}
\newcommand{\hten}{\otimes^h}
\newcommand{\ehten}{\otimes^{eh}}

\newcommand{\whten}{\otimes^{w^*h}}

\newcommand{\id}{\textnormal{id}}
\newcommand{\h}[1]{\widehat{#1}}

\newcommand{\Irr}{\mathrm{Irr}(\mathbb{G})}

\providecommand{\norm}[1]{\lVert#1\rVert}

\newcommand{\A}{\mathcal{A}}

\newcommand{\G}{\mathbb{G}}

\newcommand{\C}{\mathbb{C}}
\newcommand{\N}{\mathbb{N}}

\newcommand{\Amod}{A\hskip2pt\mathrm{mod}}
\newcommand{\modA}{\mathrm{mod}\hskip2pt A}

\newcommand{\bI}{I}

\newcommand{\ignore}[1]{{  }}

\begin{document}

\title[]{A new duality via the Haagerup tensor product}
\author{Mahmood Alaghmandan}
\email{mahmood.alaghmandan@carleton.ca}
\author{Jason Crann}
\email{jason.crann@carleton.ca}
\author{Matthias Neufang}
\email{matthias.neufang@carleton.ca}
\address{School of Mathematics and Statistics, Carleton University, Ottawa, ON, Canada H1S 5B6}


\begin{abstract} We initiate the study of a new notion of duality defined with respect to the module Haagerup tensor product. This notion not only recovers the standard operator space dual for Hilbert $C^*$-modules, it also captures quantum group duality in a fundamental way. We compute the so-called Haagerup dual for various operator algebras arising from $\ell^p$ spaces. In particular, we show that the dual of $\ell^1$ under \textit{any} operator space structure is $\min\ell^\infty$. In the setting of abstract harmonic analysis we generalize a result of Varopolous by showing that $C(\G)$ is an operator algebra under convolution for any compact Kac algebra $\G$. We then prove that the corresponding Haagerup dual $C(\G)^h=\ell^\infty(\h{\G})$, whenever $\h{\G}$ is weakly amenable. Our techniques comprise a mixture of quantum group theory and the geometry of operator space tensor products.
\end{abstract}

\begin{spacing}{1.0}

\maketitle

\section{Introduction}

Let $A$ be a completely contractive Banach algebra. We investigate a new form of duality by considering
$$A^h:=(A\hten_A A)^*,$$
where $\hten_A$ is the module Haagerup tensor product. For a $C^*$-algebra, this notion recovers the standard operator space dual $A^h=A^*$, while for the convolution algebra $\LO$ of a locally compact group $G$, we have $\LO^h=M_{cb}A(G)$ \cite{Gil}, the algebra of completely bounded multipliers of the Fourier algebra $A(G)$. From the quantum group perspective, $A(G)$ is the dual object to $\LO$. Indeed, for abelian groups, $A(G)\cong L^1(\h{G})$, where $\h{G}$ is the Pontryagin dual of $G$. Thus, on the one hand, for a $C^*$-algebra -- the prototypical operator space --  our notion recovers the operator space dual, while for convolution algebras $\LO$, it is sensitive enough to detect and reveal the underlying quantum group structure, reflected through quantum group duality. The Haagerup tensor product is instrumental in this regard: if one instead considers the operator space projective tensor product $\pten$, then for any $C^*$-algebra $A^{\wedge}:=(A\pten_A A)^*\cong A^*$, and for any locally compact group $G$, $\LO^{\wedge}:=(\LO\pten_{\LO}\LO)^*\cong(\LO)^*=\LI$. 

Building on the foundational work of Rieffel \cite{Ri}, Cigler, Losert and Michor \cite{CL} studied this type of duality in the setting of Banach space projective tensor products, via
$$X^\gamma:=\mc{B}_A(X,A^*)=(X\ten^{\gamma}_A A)^*,$$
where $X$ is a Banach module over a Banach algebra $A$. Although this notion differs from the algebraic notion of dual module, it is a natural object to consider from the functional analytic perspective. In particular, it coincides with the Banach space dual when $A=\C$. This type of duality was further studied by Grosser, who computed the dual modules for a variety of examples \cite{G}. 

An analogous notion of duality exists for operator modules over completely contractive Banach algebras $A$, where
$$X^{\wedge}:=\mc{CB}_A(X,A^*)=(X\pten_A A)^*.$$
In this context, it is also natural to consider the module Haagerup tensor product $\hten_A$. Indeed, the motivation above suggests that the ``Haagerup dual'' of $X$, 
$$X^h:=(X\hten_A A)^*,$$
detects a finer, intrinsic form of duality.


After a preliminary section we present the general notion of Haagerup duality in Section \ref{s:3} and establish a few basic tools we will use in the sequel. Section \ref{s:l-p} concerns the computation of Haagerup duals for a variety of operator space structures on $\ell^p$ spaces under pointwise product, which may be seen as Haagerup dual analogues of a variety of results from \cite{CL,G}. In particular, we prove that    
$$(o\ell^1(\bI))^h=\min\ell^\infty(\bI),$$
regardless of the operator algebra structure $o\ell^1(\bI)$. In stark contrast, we show that the Haagerup duals of $\ell^2(I)$ depend crucially on the given operator space structure. Here, our techniques rely on the operator space geometry of $\ell^p$ spaces from \cite{BL2}.

By a classical result of Varopolous \cite{V}, the space $C(G)$ of continuous functions on a compact group $G$ is an operator algebra under convolution. It is also known to experts that the reduced group $C^*$-algebra $C^*_\lm(G)$ of a discrete group $G$ is an operator algebra under pointwise multiplication. In Section \ref{s:5} we unify these results by showing that $C(\G)$ is an operator algebra under convolution for any compact Kac algebra $\G$. We then prove that the corresponding Haagerup dual satisfies
$$C(\G)^h=\ell^\infty(\h{\G}),$$
whenever $\h{\G}$ is weakly amenable. In particular, for any weakly amenable discrete group $G$, 
$$C^*_\lm(G)^h=\ell^\infty(G).$$
Here, again, we see a reflection of quantum group duality through our new notion. 

Viewing $C^*_\lm(G)$ has an operator module over the Fourier algebra $A(G)$, we finish Section \ref{s:5} by showing that for any weakly amenable discrete group $G$, the corresponding dual $C^*_\lm(G)^h=\ell^2_r(G)$. Thus, we see that the underlying module structure can have a significant impact on the operator space geometry of the dual. In this section our techniques rely on Peter--Weyl theory, the extended Haagerup tensor product, and the operator space structure of column and row Hilbert spaces.

Several natural and interesting questions are generated by this paper, such as analogues of reflexivity, etc., which will be postponed to future work.

\section{Notation and Preliminaries}

Throughout the paper we adopt standard notation from operator space theory, referring the reader to \cite{ER} for details. We let $\ten$ and $\pten$ denote the algebraic and operator space projective tensor products, respectively. On a Hilbert space $H$, we let $\mc{K}(H)$, $\mc{T}(H)$ and $\BH$ denote the spaces of compact, trace class, and bounded operators, respectively. We also let $H_r$ and $H_c$ denote the row and column operator space structures on $H$, respectively. 

For a linear map $\varphi :  X \to  Y$  and $n\in \mathbb{N}$, we let
$\varphi^{(n)} : M_n(  X)\to M_{n}( Y)$ denote its $n^{th}$ amplification given by
$$\varphi^{(n)}\left([x_{ij}]\right) = \left[ \varphi(x_{ij})\right].$$
Given operator spaces $X$ and $Y$ the Haagerup tensor norm of $ u=[u_{ij}] \in M_n(X\otimes Y)$ is defined by 
\[
\displaystyle \Vert u\Vert _h = \inf
\{\Vert x\Vert _{n,k}\ \Vert y\Vert _{k,n}\},
\]
where the infimum is taken over $ k \in {\mathbb{N}}$,  $ x=[x_{ij}]  \in M_{n,k}(X)$,  $ y=[y_{ij}] \in M_{k,n}(Y)$ such that $ u$ is the matrix inner product
\[
u= x\odot y= \left[ \sum_{\ell=1}^k x_{i\ell}\otimes y_{\ell j} \right]
\] 
The Haagerup tensor product $ X \otimes^h Y$  is the completion of $ X \otimes Y $ with respect to this operator space tensor norm.
For $C^*$-algebras $A$ and $B$, the Haagerup norm on $A\ten B$ satisfies
\[
\| u\|_h := \inf \left\{
\left\| \sum_{k=1}^n a_k a_k^* \right\|^{\frac{1}{2}} \left\| \sum_{k=1}^n b_k^* b_k \right\|^{\frac{1}{2}}: \;  
u = \sum_{k=1}^n a_k \otimes b_k \right\}.
\] 

The \emph{extended Haagerup tensor product} $X \otimes^{eh} Y$ of operator spaces $X$ and $Y$ is defined by the subspace of $(X^* \otimes^h Y^*)^*$ corresponding to the completely bounded bilinear maps from $X^* \times Y^*\rightarrow \mathbb{C}$ which are separately weak* continuous. If $X$ and $Y$ are duals of operator spaces $X_*$ and $Y_*$, respectively, then $X \otimes^{eh} Y$ coincides with the \emph{weak* Haagerup tensor product} \cite{BS}. In this case, 
\[
X \otimes^{eh} Y = (X_* \otimes^h Y_*)^*=X\whten Y
\]
complete isometrically. 

A liner mapping $\vphi:X\rightarrow Y$ factors through a row Hilbert space if there is a Hilbert space $H$ and completely bounded maps $r:X\rightarrow H_r$ and $s:H_r\rightarrow Y$ for which the following diagram commutes.

\begin{equation*}
\begin{tikzcd}
&H_r \arrow[rd, "s"]\\
X \arrow[ru, "r"] \arrow[rr, "\vphi"] & & Y
\end{tikzcd}
\end{equation*}

We let $\Gam^{2,r}(X,Y)$ denote the space of such mappings. Given $\vphi=[\vphi_{ij}]\in M_n(\Gam^{2,r}(X,Y))$, the associated mapping $\vphi:X\rightarrow M_n(Y)$ satisfies

\begin{equation*}
\begin{tikzcd}
&M_{n,1}(H_r) \arrow[rd, "s_{n,1}"]\\
X \arrow[ru, "r"] \arrow[rr, "\vphi"] & & M_n(Y)
\end{tikzcd}
\end{equation*}
where $s:H_r\rightarrow M_{1,n}(Y)$. The norm $\gamma^{2,r}(\vphi)=\inf\{\norm{r}_{cb}\norm{s}_{cb}\}$ then determines an operator space structure on $\Gam^{2,r}(X,Y)$. It is well-known that $\Gam^{2,r}(X,Y^*)\cong(X\ten^h Y)^*$ completely isometrically, via
\begin{equation}\label{e:Haa}\la\vphi, x\ten y\ra=\la\vphi(x),y\ra, \ \ \ x\in X, \ y\in Y.\end{equation}
A similar construction exists for column Hilbert spaces, and one has $\Gam^{2,c}(X,Y^*)\cong(Y\hten X)^*$.

A Banach algebra $A$ equipped with an operator space structure said to be
\emph{completely contractive}  if the mapping 
$$m_A:A\pten A\ni a\ten b\mapsto ab\in A$$
is completely contractive. It is known that a completely contractive Banach algebra is completely isomorphic to an operator algebra if and only if the mapping $m_A$ extends to a completely bounded map on ${A}\otimes_h {A}$ \cite[Theorem~5.2.1]{BL}.  

An operator space $X$ is a \emph{left operator $A$-module} if it is a left Banach $A$-module for which the module map
$$m_X:A\pten X\rightarrow X$$
is completely contractive (\cite[3.1.3]{BL}). We let $\Amod$ denote the category of left operator $A$-modules with completely bounded left $A$-module maps. We say that $X\in\Amod$ is \textit{essential} if $\la A\cdot X\ra=X$, where $\la\cdot\ra$ denotes the closed linear span. An object $X\in\Amod$ is said to be a left \textit{$h$-module} if $m_X$ extends to a complete contraction on $A\hten X$. Right operator $A$-modules and operator $A$-bimodules are defined similarly, along with the corresponding $h$-module notions. 

If $X$ and $Y$ are right and left operator $A$-modules, respectively, then the \textit{module tensor product} of $X$ and $Y$ is the quotient space $X\pten_{A}Y:=X\pten Y/N$, where
\begin{equation}\label{e:1}N=\la x\cdot a\ten y-x\ten a\cdot y\mid x\in X, \ y\in Y, \ a\in A\ra.\end{equation}
We also have the \emph{module Haagerup tensor} product $X \otimes^h_A Y=X\hten Y/N_h$, where $N_h$ is defined as in (\ref{e:1}) but with the appropriate norm closure. It follows that
\begin{equation}\label{eq:eh-A-module}
(X \otimes^h_A Y)^* = _A\Gamma^{2,c}(Y,X^*) = \Gamma^{2,r}_A(X,Y^*)
\end{equation}
where $_A\Gamma^{2,c}(Y,X^*)$ is the subspace of left $A$-module maps in $\Gamma^c(Y,X^*)$ and similarly $\Gamma^{2,r}_A(X,Y^*)$ is the subspace of right $A$-module maps in $\Gamma^{2,r}(X,Y^*)$.

\section{Haagerup Duality for Operator Modules}\label{s:3}

\begin{defn}\label{d:Haagerup-dual}
Let ${A}$ be a completely contractive Banach algebra and $X$ be in $\Amod$. The \emph{Haagerup dual} of $X$, denoted $X^h$, is the dual space $(A \otimes^h_A X)^*$ furnished with the dual operator space structure. A similar definition applies to right modules.
\end{defn}

Note that  applying \eqref{eq:eh-A-module}, we immediately have completely isometric identifcations
\begin{equation}\label{eq:H-A-X-cr}
X^h = _A\Gamma^{2,c}(X, A^*) = \Gamma^{2,r}_A (A, X^*)
\end{equation} 
Moreover, since $\Gamma^{2,r}$ is a mapping space, $X^h$ inherits a canonical right $A$-module structure, given by 
$$(\vphi\cdot a)(b)=\vphi(b)\cdot a, \ \ \ a\in A, \ \vphi\in X^h.$$
One may therefore view $(\cdot)^h$ as a contravariant functor $\Amod\rightarrow\modA$, where a morphism $f:X\rightarrow Y$ transforms to $f^h:Y^h\ni\vphi\mapsto f^*\circ\vphi\in X^h$.

\begin{example} As mentioned above, one of the motivations for considering this notion comes from quantum group duality. In \cite{ACN} it was shown that for any 
locally compact quantum group $\G$ whose dual $\widehat{\G}$ is either QSIN or has bounded degree, then 
$$L^1(\G)^h\cong M_{cb}(\LOQH)$$ 
completely isomorphically, and weak*-weak* homeomorphically. Thus, for a large class of quantum group convolution algebras $\LOQ$, including $\LO$ for any locally compact group $G$, the Haagerup dual $L^1(\G)^h$ captures the dual object $\h{\G}$ in a fundamental way.
\end{example}

The following tool will be used heavily in the sequel when computing specific examples. The proof technique is similar to \cite[Corollary 7.4]{C}.

\begin{prop}\label{p:h-module-mb-ai}
Let ${A}$ be a completely contractive Banach algebra with a cb-multiplier bounded approximate identity. If $X\in\Amod$ is a left $h$-module, then $\mathrm{Ker}(m_X)=N_X$. In particular, if $m_X: {A} \otimes_h X \rightarrow \mathrm{Im}(m_X)$ is a complete quotient map onto its image, then $X^h\cong\mathrm{Im}(m_X)^*$ completely isometrically and weak*-weak* homeomorphically.
\end{prop}

\begin{proof}
Let $(e_\alpha)_\alpha$ be a cb-multiplier bounded approximate identity for $A$, that is, $\norm{e_\alpha a - a}_{{A}}\rightarrow 0$ for all $a\in A$ and $\sup_\alpha \norm{e_\alpha}_{M_{cb}{A}}<\infty$. As $N_X \subseteq \mathrm{Ker}(m_X)$, where
$$N_X:=\la a\cdot b\ten x-a\ten b\cdot x\mid a,b\in A, \ x\in X\ra,$$
multiplication induces a canonical map $\widetilde{m_X}:A\hten_A X\rightarrow X$. Letting $l_a:A\ni b\mapsto ab\in A$ denote left multiplication by $a\in A$, on elementary tensors we have
$$(l_a\ten \id)(b\ten x)+N_X=a\ten b\cdot x+N_X=a\ten m_X(b\ten x)+N_X, \ \ \ a,b\in A, \ x\in X.$$
It follows that $(l_a\ten\id)(y)+N_X=a\ten m_X(y)+N_X$ for all $a\in A$ and $y\in A\hten X$. Thus, if $y \in \mathrm{Ker}(m_X)$ we have
$$y+N_X=\lim_\alpha(l_{e_\alpha}\ten \id)y+N_X=\lim_{\alpha}e_\alpha\ten m_X(y)+N_X=N_X.$$
Thus, $N_X=\mathrm{Ker}(m_X)$. If, in addition, $m_X: {A} \otimes_h X \rightarrow \mathrm{Im}(m_X)$ is a complete quotient map, then $A\hten_A X\cong\mathrm{Im}(m_X)$ completely isometrically, and so $X^h\cong\mathrm{Im}(m_X)^*$ completely isometrically and weak*-weak* homeomorphically.
\end{proof}

\begin{cor}\label{c:bai}
Let ${A}$ be a completely contractive Banach algebra with a bounded approximate identity and let $X\in\Amod$ be an essential left $h$-module. Then $X^h\cong X^*$ completely isomorphically and weak*-weak* homeomorphically. In particular, if $A$ is completely isomorphic to an operator algebra with a bounded approximate identity (e.g., a $C^*$-algebra) then $A^h\cong{A}^*$.
\end{cor}

\begin{proof}
Let $(e_\alpha)$ be a bounded approximate identity for $A$ with $\sup_\alpha\norm{e_\alpha}\leq C$. It follows that multiplication induces a complete isomorphism $m_X:A\pten_AX\cong \la A\cdot X\ra=X$, where the $cb$-norm of the inverse map is bounded by $C$ \cite[Proposition 6.4]{D}. Since $X$ is a left $h$-module, from Proposition \ref{p:h-module-mb-ai} and the commutative diagram

\begin{equation*}
\begin{tikzcd}
A\pten_A X\arrow[rr, "\widetilde{m_X}"]\arrow[rd] &   &X \\
&A\hten_A X \arrow[ru, "\widetilde{m_X}"] &
\end{tikzcd}
\end{equation*}
it follows that $\widetilde{m_X}: {A} \otimes^h_{A} {X} \rightarrow {X}$ is a complete isomorphism with the same bound on the $cb$-inverse. In particular, if $A$ is completely isomorphic to an operator algebra then by \cite[Theorem~5.2.1]{BL}, it is a left $h$-module over itself, so $A^h\cong A^*$.
\end{proof}


\section{$\ell^p$-spaces as operator algebras}\label{s:l-p}

In this section we compute the Haagerup duals of $\ell^p$-spaces seen as operator algebras under pointwise multiplication. It is known that for any index set $I$, the space $\ell^1(\bI)$ becomes an operator algebra under pointwise multiplication with \textit{any} operator space structure, say $o\ell^1(\bI)$. Moreover,
\[
m^{o}: o\ell^1(\bI) \otimes^h o\ell^1(\bI) \rightarrow o\ell^1(\bI)
\]
is a complete contraction \cite[Theorem 3.1]{BL2}. Also, the space $\max \ell^p(\bI)$ is an operator algebra under pointwise multiplication if and only if $1\leq p\leq 2$ \cite[Theorem 3.4]{BL2}. 

\begin{prop}\label{p:l-1}
For any operator space structure $o\ell^1(\bI)$,  $o\ell^1(\bI)^h=\min\ell^\infty(\bI)$.
\end{prop}

\begin{proof} We first prove the case for $o\ell^1(\bI)=\max\ell^1(\bI)$. Recall that $\hat{m} :\max\ell^1(\bI) \pten \max \ell^1(\bI) \rightarrow \max\ell^1(\bI)$ is a complete quotient map (see \cite[Example~5.7]{D}, for instance).
Clearly, the following diagram commutes.
 \[
 \xymatrix{
\max \ell^1(\bI) \pten \max \ell^1(\bI)  \ar[rd]_{\hat{m}} \ar[r]^{\iota} & \max \ell^1(\bI)\otimes^h \max \ell^1(\bI) \ar[d]^{m^{\max}} \\
       & \max\ell^1(\bI)       }
 \]
Since the inclusion $\iota$ is a complete contraction, it follows that $m^{max}$ is a complete quotient map. Consequently, applying Proposition~\ref{p:h-module-mb-ai},  
 \[
(\max \ell^1(\bI))^h=(\max\ell^1(I))^*=\min\ell^\infty(\bI).
 \]
Now for an arbitrary operator space structure $o\ell^1(\bI)$, let $i: \max\ell^1(\bI) \rightarrow  o \ell^1(\bI)$ and $i': o\ell^1(\bI) \rightarrow \min\ell^1(\bI)$ be the canonical complete contractions. In \cite[Theorem 3.1]{BL2}, it was shown that 
\[
\bar{m}: \min\ell^1(\bI) \otimes^h \min \ell^1(\bI) \rightarrow \max \ell^1(\bI)
\]
is a complete contraction, where $\bar{m}$ is the extension of pointwise multiplication.
One can easily show that the following diagram commutes.
\[
 \xymatrix{
\max \ell^1(\bI)  \otimes^h \max \ell^1(\bI)  \ar[rd]_{ {m}^{\max}} \ar[r]^{\quad i \otimes i} & o \ell^1(\bI)\otimes^h o \ell^1(\bI) \ar[r]^{ i'\otimes i'} & \min\ell^1(\bI)\otimes^h\min \ell^1(\bI)\ar[ld]^{\bar{m}} \ \\
       & \max\ell^1(\bI)     }
 \]
Now let $u \in M_n(\max \ell^1(\bI))$ for some positive integer $n$ with $\norm{u}_n < 1$. Since $m^{\max}$ is a complete quotient, there is some $v \in M_n (\max \ell^1(\bI)  \otimes^h \max \ell^1(\bI))$ with $\norm{v}_n < 1$ and $(m^{\max})^{(n)}(v)=u$. But $w=(i\otimes i)^{(n)} (v)$ is an element in $M_n(o \ell^1(\bI)\otimes^h o \ell^1(\bI))$ of norm strictly less than $1$. Therefore
\[
m:=\bar{m} \circ (i'\otimes i') : o\ell^1(\bI) \otimes^h o\ell^1(\bI) \rightarrow \max \ell^1(\bI)
\]
and hence the induced map $\tilde{m}:o\ell^1(\bI) \otimes^h_{o\ell^1(\bI)} o\ell^1(\bI)\rightarrow\max\ell^1(I)$ is a complete quotient, implying that $(\tilde{m})^*:\min\ell^\infty(I)\rightarrow N_{o\ell^1(\bI)}^{\perp}$ is a complete isometry, where
$$N_{o\ell^1(\bI)}=\la f\cdot g\ten h - f\ten g\cdot h\mid f,g,h\in o\ell^1(\bI)\ra\subseteq o\ell^1(\bI)\hten o\ell^1(\bI).$$
Let $F\in N_{o\ell^1(\bI)}^{\perp}$. Then for all $i,j\in I$, 
$$\la F,\delta_i\ten \delta_j\ra=\la F,\delta_i\cdot\delta_i\ten\delta_j\ra=\la F,\delta_i\ten\delta_i\cdot\delta_j\ra=\delta_{ij}\la F,\delta_i\ten\delta_i\ra.$$
It follows that $f:I\ni i\mapsto \la F,\delta_i\ten\delta_i\ra\in\C$ lies in $\ell^\infty(I)$ and $(\tilde{m})^*(f)=F$. Thus, $$(o\ell^1(\bI))^h=\min\ell^\infty(\bI).$$

\end{proof}

\begin{cor}\label{c:l-p}
For $1\leq p \leq 2$, $(\max \ell^p(\bI))^h=\min \ell^\infty(\bI)$.
\end{cor}

\begin{proof}
The canonical inclusions of the $\ell^p$-spaces lead to complete contractions 
\[
\xymatrix{
\max\ell^1(\bI)  \ar[r]_{\iota_1} & \max \ell^p(\bI) \ar[r]_{\iota_2} & \max \ell^2(\bI)
}
\]
such that the following diagram commutes.
\[
\xymatrix{
\max \ell^1(\bI)\otimes^h \max \ell^1(\bI) \ar[r]^{\iota_1 \otimes \iota_1}  \ar[rrd]_{m_1}     &  \max\ell^p(\bI) \otimes^h \max \ell^p(\bI) \ar[r]^{\iota_2 \otimes\iota_2} \ar[rd]^{m_p}& \max \ell^2(\bI) \otimes^h \max\ell^2(\bI)   \ar[d]^{m_2}\\
       & &  \max \ell^1(\bI)      }
\]
Since $m_1$ is a complete quotient map (by the proof of Proposition \ref{p:l-1}), it follows that
\[
m_p: \max\ell^p(\bI) \otimes^h \max \ell^p(\bI) \rightarrow \max\ell^p(\bI)
\]
is also. Proposition \ref{p:h-module-mb-ai} then entails $(\max\ell^p(\bI))^h=\min\ell^\infty(\bI)$ using the fact that $\max \ell^p(\bI)$ has a cb-multiplier bounded approximate identity.
\end{proof}

In Proposition \ref{p:l-1} we saw that $(o\ell^1(\bI))^h=\min\ell^\infty(\bI)$, regardless of the operator space structure imposed on $\ell^1(I)$. In stark contrast, we now show that different operator space structures on $\ell^2(I)$ yield very different Haagerup duals.

\begin{prop}\label{p:OH} Let $OH(\bI)$ denote the operator Hilbert space structure on $\ell^2(\bI)$ for an index set $\bI$. Then $(OH(\bI))^h=OH(\bI)^*$. 
\end{prop}

\begin{proof} It is known that $\overline{OH(\bI)}=OH(\bI)^*$, $OH(\bI) \otimes^h OH(\bI) = OH(\bI \times \bI)$ (see \cite[Chapter 7]{P}), and that $OH(\bI)$ is an operator algebra under pointwise multiplication, that is, the multiplication map 
\[
m : OH(\bI) \otimes^h OH(\bI)= OH( \bI \times \bI) \rightarrow OH(\bI) 
\]
is completely bounded \cite[Theorem~2.1]{BL2}.  Here we prove that in fact $m$ is a complete quotient. To do so, note that 
\[
m^*: OH(\bI)^* \rightarrow OH(\bI \times \bI)^*
\]
where $m^*(\bar{\delta}_i)=\bar{\delta}_i \otimes \bar{\delta}_i$ for every $i \in \bI$. Let $\xi = \sum_{i \in \bI(\xi)} x_i \otimes \bar{\delta}_i$ be an element in the algebraic tensor product $\mc{K}(\ell^2) \otimes  \overline{OH(\bI)}$. By \cite[Chapter 7]{P} we have
\[
\norm{\xi}_{OH(\bI)^*}= \left\| \sum_{i\in \bI}  x_i \otimes \bar{x}_i\right\|_{\min}^{\frac{1}{2}}.
\]
Therefore,

$$\norm{(m^*)^{(\infty)}(\xi)}_{M_{\infty}(OH(\bI \times \bI)^*)}
= \left\| \sum_{i \in \bI}  x_i \otimes \bar{\delta}_i \otimes \bar{\delta}_i\right\|_{M_{\infty}(OH(\bI \times \bI)^*)}
=\left\| \sum_{i \in \bI}  x_i \otimes \bar{x}_i \right\|_{\min}^{\frac{1}{2}} 
=\norm{\xi}_{M_{\infty}(OH(\bI)^*)},$$
and $m^*$ is a complete isometry. Consequently, $m: OH(\bI)\otimes^h OH(\bI) \rightarrow OH(\bI)$ is a complete quotient map.

Now, on the first matrix level $OH(\bI) \otimes^h OH(\bI)=\ell^2(I\times I)$. Thus, if $(\chi_{F})$ is the net of characteristic functions of finite subsets $F\subseteq I$, then for any $y\in OH(\bI) \otimes^h OH(\bI)$, 
$$\norm{(\chi_F\ten 1)y-y}_{OH(\bI) \otimes^h OH(\bI)}\leq\norm{(\chi_F\ten \chi_F)y-y}_{OH(\bI) \otimes^h OH(\bI)}\rightarrow0.$$
It follows from the proof of Proposition \ref{p:h-module-mb-ai} that $(OH(\bI))^h=OH(\bI)^*$. 
\end{proof}

\begin{prop}\label{p:C-R}
$(\ell^2_r(\bI))^h=(\ell^2_r(\bI))^*$ and $(\ell^2_c(\bI))^h= (\ell^2_c(I))^*$.
\end{prop}

\begin{proof}
First let us treat the case of $\ell^2_r(\bI)$. It is shown in the proof of \cite[Theorem~2.1]{BL2} that the multiplication mapping 
\[
m_r: \ell^2_r(\bI) \otimes^h \ell^2_r(\bI) \rightarrow \ell^2_r(\bI)
\]
is a complete contraction. We prove that $m_r$ is in fact a complete quotient. Let $n\in\N$ and $\xi = [\xi_{ij}]\in M_{n}(\ell^2_r(\bI))$ so that $\norm{\xi}_{M_{n}(\ell^2_r(\bI))}< 1$.
For each pair $i,j$, $\xi_{ij}=(\alpha_s^{(i,j)})_{s \in \bI}$ where $\alpha_s^{(i,j)} \in \Bbb{C}$. 
 Note that $\xi$ is nothing but $(m_r)^{(n)}(\eta)$ for
 \begin{equation}\label{eq-input-r}
\eta= \left(  \sum_{i,j=1}^{n} e_{ij} \otimes \left[ \ldots \alpha_s^{(i,j)}\delta_s \ldots\right]_{s \in \bI} \right) \odot \left( I_{n\times n} \otimes \left[
\begin{array}{c}
\vdots\\
\delta_s\\
\vdots
\end{array} 
\right]_{s\in \bI}\right)
\end{equation}
where $e_{ij}$ are the canonical matrix units of $M_{n}$ and $I_{n\times n}$ is the $n \times n$ identity matrix. 
All we need to show is that the norm of \eqref{eq-input-r} is less than $1$. To do so, note that 
\begin{eqnarray*}
\left\| \sum_{i,j=1}^{n} e_{ij} \otimes \left[ \ldots \alpha_s^{(i,j)}\delta_s \ldots\right]_{s \in \bI} \right\|_{M_{n}(M_{1,\infty}(\ell^2_r(\bI)))} 
&=&  \left\| \sum_{i,j=1}^n \sum_{k=1}^n  e_{ik} e_{jk} \otimes \langle \xi_{ik}, \xi_{jk}\rangle   \right\|^{\frac{1}{2}}_{M_{n}}\\
&=&  \left\|  \sum_{i,j=1}^n \sum_{k=1}^n  e_{ij} \otimes \langle \xi_{ik}, \xi_{jk}\rangle   \right\|^{\frac{1}{2}}_{M_{n}}\\
&=&  \left\| \left[ \sum_{k=1}^n  \langle \xi_{ik}, \xi_{jk}\rangle       \right]_{i,j=1}^n \right\|^{\frac{1}{2}}_{M_{n}} = \norm{\xi}_{M_{n}(\ell^2_r(\bI))}.
\end{eqnarray*}
Also,
$$
\left\| I_{n\times n} \otimes \left[
\begin{array}{c}
\vdots\\
\delta_s\\
\vdots
\end{array} 
\right]_{s\in \bI} \right\|_{M_{n}(M_{\infty, 1}(\ell^2_r(\bI)))}  
= \left\|  
\left[
\begin{array}{c}
\vdots\\
\delta_s\\
\vdots
\end{array} 
\right]_{s\in \bI}
 \right\|_{M_{\infty, 1}(\ell^2_r(\bI))}  \\
 = \left\| \;  
  \big[
\langle \delta_s, \delta_t \rangle
\big]_{s, t \in \bI}
 \right\|^{\frac{1}{2}}_{M_\infty} = \left\| I_{\infty} \right\|_{M_\infty}=1
$$
where $I_\infty$ is the identity matrix.  Therefore, $\norm{\eta}_{M_{n}(\ell^2_r(\bI) \otimes^h \ell^2_r(\bI))}< 1$. This implies that $m_r$ is a complete quotient map. Since $\ell^2_r(\bI)$ has a cb-multiplier bounded approximate identity (see \cite[Proposition 3.4.2]{ER}, Proposition~\ref{p:h-module-mb-ai} implies that $(\ell^2_r(\bI))^h=(\ell^2_r(\bI))^*$.

The case $m_c:  \ell^2_c(\bI) \otimes^h \ell^2_c(\bI) \rightarrow \ell^2_c(\bI)$ is similar where for each $\xi=[\xi_{ij}]\in M_{n}(\ell^2_c(\bI))$ with $\norm{\xi}_{M_{n}(\ell^2_c(\bI))} < 1$, we have $\xi= m_c^{(n)}(\eta)$ where
\[
\eta= \left(  I_{n\times n}  \otimes \left[ \ldots \delta_s \ldots\right]_{s \in \bI} \right) \odot \left( \sum_{i,j=1}^{n}e_{ij}  \otimes \left[
\begin{array}{c}
\vdots\\
\alpha_s^{(i,j)} \delta_s\\
\vdots
\end{array} 
\right]_{s\in \bI}\right)
\]
so that for each pair $i,j$, $\xi_{ij}=(\alpha_s^{(i,j)})_{s \in \bI}$ as an element in $\ell^2(\bI)$.  Then one can show that $\norm{\eta}_{M_{n}(\ell^2_c(\bI) \otimes^h \ell^2_c(\bI))}<1$. 
\end{proof}

\section{Examples from Abstract Harmonic Analysis}\label{s:5}

A \textit{locally compact quantum group} is a quadruple $\G=(\LIQ,\Gam,\vphi,\psi)$, where $\LIQ$ is a Hopf-von Neumann algebra with a co-associative co-multiplication $\Gam:\LIQ\rightarrow\LIQ\oten\LIQ$, and $\vphi$ and $\psi$ are fixed (normal faithful semifinite) left and right Haar weights on $\LIQ$, respectively \cite{KV1,KV2}. For every locally compact quantum group $\G$ there exists a \e{left fundamental unitary operator} $W$ on $L^2(\G,\vphi)\ten L^2(\G,\vphi)$ and a \e{right fundamental unitary operator} $V$ on $L^2(\G,\psi)\ten L^2(\G,\psi)$ implementing the co-multiplication $\Gam$ via
\begin{equation*}\Gam(x)=W^*(1\ten x)W=V(x\ten 1)V^*, \ \ \ x\in\LIQ.\end{equation*}
Both unitaries satisfy the \e{pentagonal relation}; that is,
\begin{equation}\label{penta}W_{12}W_{13}W_{23}=W_{23}W_{12}\hs\hs\mathrm{and}\hs\hs V_{12}V_{13}V_{23}=V_{23}V_{12}.\end{equation}
At the level of the Hilbert spaces,
$$W^*\Lphis(x\ten y)=\Lphis(\Gam(y)(x\ten 1))\hs\mathrm{and}\hs V\Lambda_{\psi\ten\psi}(a\ten b)=\Lambda_{\psi\ten\psi}(\Gam(a)(1\ten b))$$
for $x,y\in\mc{N}_\vphi$ and $a,b\in\mc{N}_\psi$. By \cite[Proposition 2.11]{KV2}, we may identify $L^2(\G,\vphi)$ and $L^2(\G,\psi)$, so we will simply use $\LTQ$ for this Hilbert space throughout the paper. The \textit{reduced quantum group $C^*$-algebra} of $\LIQ$ is defined as
$$C_0(\G):=\overline{\{(\id\ten\om)(W)\mid \om\in\TCQ\}}^{\norm{\cdot}}.$$
We say that $\G$ is \textit{compact} if $C_0(\G)$ is a unital $C^*$-algebra, in which case we denote $C_0(\G)$ by $C(\G)$. For compact quantum groups it follows that $\vphi$ is finite and right invariant. In particular, $\vphi=\psi$, and we may normalize $\vphi$ to a state on $\LIQ$. If, in addition, $\vphi$ is tracial, we say that $\G$ is a \textit{compact Kac algebra}. In what follows we assume that $\G$ is a such a quantum group. 

We let $R$ denote the \textit{antipode} $\G$. We will make use of the following relations
\begin{equation*}(\id\ten\vphi)(\Gam(x)(1\ten y)))=R((\id\ten\vphi)((1\ten x)\Gam(y))), \ \ \ R((\vphi\ten\id)(x\ten 1)\Gam(y))=(\vphi\ten\id)(\Gam(x)(y\ten 1)),\end{equation*}
which are valid for all $x,y\in\LIQ$.
 
Let $\LOQ$ denote the predual of $\LIQ$. Then the pre-adjoint of $\Gam$ induces an associative completely contractive multiplication on $\LOQ$, defined by
\begin{equation*}\star:\LOQ\pten\LOQ\ni f\ten g\mapsto f\star g=\Gam_*(f\ten g)\in\LOQ.\end{equation*}
There is a canonical $\LOQ$-bimodule structure on $\LIQ$, given by
\begin{equation*}\la f\star x,g\ra=\la x,g\star f\ra\hs\mathrm{and}\hs\la x\star f,g\ra=\la x,f\star g\ra, \ \ \ x\in\LIQ, \ f,g\in\LOQ.\end{equation*}
There is an involution on $\LOQ$ given by $f^o=f^*\circ R$, such that $\LOQ$ becomes a Banach *-algebra under its canonical norm. 

The restricted co-multiplication maps $C(\G)$ into $C(\G)\ten_{\min} C(\G)$, and induces a completely contractive Banach algebra structure on the operator space dual $M(\G):=C(\G)^*$. Restriction then identifies $\LOQ$ with a norm-closed two-sided ideal in $M(\G)$. 

A \textit{unitary co-representation} of $\G$ is a
unitary $U\in\LIQ\oten\BH$ satisfying $(\Gam\ten\id)(U)=U_{13}U_{23}$. Every unitary co-representation gives rise to a $*$-representation of $\LOQ$ via
$$\LOQ\ni f \mapsto (f\ten\id)(U)\in\BH.$$
In particular, the left fundamental unitary $W$ gives rise to the \textit{left regular representation} $\lm:\LOQ\rightarrow\BLTQ$ defined by $\lm(f)=(f\ten\id)(W)$, $f\in\LOQ$,
which is an injective, completely contractive $*$-homomorphism from $\LOQ$ into $\BLTQ$. Then $\ell^{\infty}(\h{\G}):=\{\lm(f) : f\in\LOQ\}''$ is the von Neumann algebra associated with the dual (discrete) quantum group $\h{\G}$ of $\G$. Every irreducible co-representation $u^{\alpha}$ of $\G$ is finite-dimensional and is unitarily equivalent to a sub-representation of $W$, and every unitary co-representation of $\G$ can be decomposed into a direct
sum of irreducible co-representations. We let $\Irr:=\{u^{\alpha}\}$ denote a complete set of representatives
of irreducible co-representations of $\G$ which are pairwise inequivalent. Slicing by vector functionals $\om_{ij}=\om_{e_j,e_i}$ relative to an orthonormal basis
of $H_\alpha$, we obtain elements $u^{\alpha}_{ij}=(\id\ten\om_{ij})(u^{\alpha})\in\LIQ$ satisfying
\begin{equation*}\Gam(u^{\alpha}_{ij})=\sum_{k=1}^{n_\alpha}u^{\alpha}_{ik}\ten u^{\alpha}_{kj}, \ \ \ R(u_{ij}^\al)=u_{ji}^{\al^*} \ \ \ 1\leq i,j\leq n_{\alpha}\end{equation*}
The linear space $\A:=\textnormal{span}\{u^{\alpha}_{ij}\mid\al\in\Irr \ 1\leq i,j\leq n_\al\}$ forms unital Hopf *-algebra which is dense in $C(\G)$. For every $\al$ there exists a conjugate representation $\overline{\al}\in\Irr$ on $\overline{H}_\al$, such that
$u_{ij}^{\overline{\al}}=u_{ij}^{\al^*}$ (see \cite[Proposition 1.4.6]{NT2}).

The Peter--Weyl orthogonality relations for compact Kac algebras are as follows:
$$\vphi((u^\be_{kl})^*u_{ij}^\al)=\frac{\delta_{\al\be}\delta_{ik}\delta_{jl}}{n_\al}=\vphi(u^\be_{kl}(u_{ij}^\al)^*).$$
From this it follows that $\{\sqrt{n_\al}\Lphi(u_{ij}^\al)\mid\al\in\Irr, 1\leq i,j\leq n_\al\}$ is an orthonormal basis for $\LTQ$.

For an element $x\in\LIQ$, we let $x\cdot\vphi$ denote the element in $\LOQ$ given by $\la x\cdot\vphi,y\ra=\vphi(yx)=\vphi(xy)$, $y\in\LIQ$. If $x=u_{ij}^\al$, we denote $u_{ij}^\al\cdot\vphi$ by $\vphi_{ij}^\al$. By the density of $\A$ in $C(\G)$, it follows that $\LIQ\cdot\vphi$ is dense in $\LOQ$. 



As in the case of compact groups, the irreducible characters of $\G$ play an important role in the harmonic analysis. For $\al\in\Irr$, we let
$$\chi^\al:=(\id\ten\tr)(u^{\al})=\sum_{i=1}^{n_\al}u^\al_{ii}\in\LIQ$$
be the \emph{character} of $\al$. The characters (as well as the quantum characters) satisfy the decomposition relations:
\begin{equation}\label{e:N}\chi^\al\chi^\be=\sum_{\gamma\in\Irr}N^{\gamma}_{\al\be}\chi^\gamma,\end{equation}
where $N^{\gamma}_{\al\be}$ is the multiplicity of $\gamma$ in the tensor product representation $\al\ten\be$ (see \cite[Proposition 1.4.3]{NT2}). It follows that $\chi^{\overline{\al}}=\chi^{\al^*}$, $\al\in\Irr$. Letting $\vphi^\al:=\chi^\al\cdot\vphi$ be the $\LOQ$ elements corresponding to the quantum characters of $\G$, it follows from the orthogonality relations that
$$\la\vphi^\al\star f,u^{\be*}_{kl}\ra=\la f\star\vphi^\al,u^{\be*}_{kl}\ra=\la f,u^{\be*}_{kl}\ra\frac{\delta_{\al\be}}{n_\al}$$
for all $f\in\LOQ$ and $\be\in\Irr$. In particular,
\begin{equation}\label{idem}\vphi^\al\star\vphi^\al=\frac{1}{n_\al}\vphi^\al, \ \ \ \al\in\Irr.\end{equation}
As in the group setting, the irreducible characters correspond to central idempotents in the dual discrete von Neumann algebra, as shown (for instance) in \cite[Lemma 3.3]{AC}.

\begin{lem}\label{l:matrixunits} Let $\G$ be a compact Kac algebra. Then
$$\{n_\al\lm(\vphi_{ij}^\al)\mid\al\in\Irr, 1\leq i,j\leq n_\al\}$$
forms a set of matrix units for the von Neumann algebra $\ell^\infty(\h{\G})$. In particular, $z_\al:=n_\al\lm(\vphi^\al)$ are the canonical central projections implementing the isomorphism $\ell^\infty(\h{\G})=\prod_{\al\in\Irr} M_{n_\al}$.
\end{lem}

For a compact algebra $\G$, let $C_u(\G)$ be its corresponding universal $C^*$-algebra (see \cite{K} for details). There is a universal co-multiplication $\Gam_u:C_u(\G)\rightarrow C_u(\G)\ten_{\min}C_u(\G)$ satisfying $(\pi\ten\pi)\circ\Gam_u=\Gam\circ\pi$, where $\pi:C_u(\G)\rightarrow C(\G)$ is the canonical quotient map. This gives rise to a universal compact quantum group structure on $C_u(\G)$. In particular, there is a $*$-algebra $\mc{A}_u$ of universal matrix coefficients $\widetilde{u_{ij}^\al}$ which is dense in $C_u(\G)$, such that $\pi(\widetilde{u_{ij}^\al})=u_{ij}^\al$, $\al\in\Irr$, $1\leq i,j\leq n_\al$. The operator space dual $C_u(\G)^*$ inherits a canonical completely contractive Banach algebra structure, and the canonical quotient $\pi$ induces a completely isometric homomorphism from $M(\G)$ into $C_u(\G)^*$ which allows us to identify $M(\G)$ with a norm-closed two-sided ideal in $C_u(\G)^*$. 

\begin{examples}
\begin{enumerate}
\item If $G$ is a compact group, then $\G=(\LI,\Gam,\vphi)$ is a compact Kac algebra where $\Gam(f)(s,t)=f(st)$, and $\vphi$ is integration with respect to a normalized Haar measure on $G$. Here $C(\G)=C(G)$, $\LOQ=\LO$ is the classical convolution algebra, and $M(\G)=C_u(\G)^*=M(G)$ is the usual measure algebra. The dual discrete quantum group $\h{\G}$ in this case is represented by the atomic group von Neumann algebra $VN(G)=\prod_{\al\in\h{G}} M_{n_\al}$. 
\item If $G$ is a discrete group, then $\G=(VN(G),\Gam,\vphi)$ is a compact Kac algebra where $VN(G)$ is the group von Neumann algebra with co-multiplication $\Gam(\lm(t))=\lm(t)\ten\lm(t)$, and $\vphi=\la(\cdot)\delta_e,\delta_a\ra$ is the canonical trace on $VN(G)$. Here, $C(\G)=C^*_\lm(G)$, the reduced group $C^*$-algebra, $\LOQ=A(G)$ is the Fourier algebra, $M(\G)=B_\lm(G)$ is the reduced Fourier--Stieltjes, and $C_u(\G)^*=B(G)$ is the Fourier--Stieltjes algebra. 
\item Let $N\geq2$ be an integer and let $A_o(N)$ be the universal $C^*$-algebra generated by $N^2$ elements $u_{ij}$ such that the matrix $u=[u_{ij}]$ is unitary and $\overline{u}=u$, where $\overline{u}=[u_{ij}^*]$. Define $\Gam_u:A_o(N)\rightarrow A_o(N)\ten_{\min}A_o(N)$ on the generators by
$$\Gam_u(u_{ij})=\sum_{k=1}^Nu_{ik}\ten u_{kj}, \ \ \ 1\leq i,j\leq N.$$
There exists a unique Haar state $\vphi$ on $A_o(N)$ satisfying $(\id\ten\vphi)\Gam_u(x)=(\vphi\ten\id)\Gam_u(x)=\vphi(x)1$, $x\in A_o(N)$. The GNS construction $(\pi_\vphi,\xi_\vphi)$ of $\vphi$ then yields a von Neumann algebra $L^{\infty}(O_N^+):=\pi_\vphi(A_o(N))''\subseteq\mc{B}(H_\vphi)$ and a co-multiplication $\Gam: L^{\infty}(O_N^+)\rightarrow L^{\infty}(O_N^+)\oten L^{\infty}(O_N^+)$ such that $O_N^+=(L^{\infty}(O_N^+),\Gam,\om_{\xi_\vphi})$ becomes a compact Kac algebra, called the \e{free orthogonal quantum group}. In this example, $C_u(O_N^+)=A_o(N)$ and $C(O_N^+)=\pi_{\vphi}(A_o(N))\subseteq\mc{B}(H_\vphi)$. Recall that the universal \textit{commutative} $C^*$-algebra generated by $u_{ij}$ satisfying the above relations is nothing but $C(O(N))$, the commutative $C^*$-algebra of continuous functions on the orthogonal group $O(N)$.
\end{enumerate}
\end{examples}

For a compact group $G$, it is well-known that $C(G)$ is an operator algebra under convolution \cite{V}. Dually, for a discrete group $G$, it is folklore that $C^*_\lm(G)$ is an operator algebra under pointwise multiplication. We now show that this phenomena persists at the level of compact Kac algebras. 

First, $\G$ is a compact Kac algebra, then for $a,b,x\in\A$, we have
\begin{align*}\la (a\cdot\vphi)\star(b\cdot\vphi),x\ra&=a\cdot\vphi((\id\ten\vphi)(\Gam(x)(1\ten b)))=a\cdot\vphi\circ R((\id\ten\vphi)((1\ten x)\Gam(b)))\\
&=R(a)\cdot\vphi((\id\ten\vphi)((1\ten x)\Gam(b)))\\
&=\la(R(a)\cdot\vphi\ten\id)\Gam(b),x\cdot\vphi\ra.\end{align*}
Hence, the convolution product on $\LOQ$ induces a Banach algebra multiplication $C(\G)\ten^{\gamma} C(\G)\rightarrow C(\G)$ satisfying
$$a\star b = (R(a)\cdot\vphi\ten\id)\Gam(b), \ \ \ a,b\in C(\G).$$
A similar argument also works at the universal level, yielding a convolution Banach algebra structure on $C_u(\G)$. At the level of irreducible coefficients
$$u^\al_{ij}\star u^\be_{kl}=((u^\al_{ji})^*\cdot\vphi\ten\id)\Gam(u_{kl}^\be)=\frac{\delta_{\al\be}\delta_{jk}}{n_\al}u_{il}^\al.$$
The following Lemma refers to the canonical Banach algebra structure on $C_u(\G)^*$.

\begin{lem}\label{l:dualBA}
Let $\G$ be a compact quantum group. Then $C_u(\G)^*\ehten C_u(\G)^*$ is a dual completely contractive Banach algebra.
\end{lem}

\begin{proof} By \cite[\S9.1]{D}, we know that $C_u(\G)^*\ehten C_u(\G)^*$ is a completely contractive Banach algebra. Suppose $(\mu_i)$ is a bounded net converging weak* to $\mu$ in $C_u(\G)^*\ehten C_u(\G)^*=(C_u(\G)\hten C_u(\G))^*$. For $\nu\in C_u(\G)^*\ehten C_u(\G)^*$ and $x,y\in\A_u$ we have
$$\la\mu_i\star\nu,x\ten y\ra=\la\mu_i\ten\nu,\Sigma_{23}(\Gam_u(x)\ten\Gam_u(y))\ra\rightarrow\la \mu\star\nu,x\ten y\ra,$$
where $\Sigma_{23}$ is the flip map on the second and third factors. By boundedness of $(\mu_i)$ together with linearity and density of $\A_u\ten\A_u$ in $C_u(\G)\hten C_u(\G)$, it follows that $\mu_i\star\nu\rightarrow \mu\star\nu$ weak*. Similarly, $\nu\star\mu_i\rightarrow\nu\star\mu$ weak*, and $C_u(\G)^*\ehten C_u(\G)^*$ is a dual Banach algebra.
\end{proof}

\begin{prop}\label{p:opalg}
Let $\G$ be a compact Kac algebra. Then $C_u(\G)$ and $C(\G)$ are operator algebras under convolution.
\end{prop}

\begin{proof}
Let $\vphi$ and $R$ denote the Haar trace and unitary antipode on $C(\G)$, as well as their universal extensions to $C_u(\G)$. Define $*$-homomorphisms $\lm,\rho:C_u(\G)\rightarrow\BLTQ$ by
$$\lm(x)\Lphi(y)=\Lphi(xy), \ \ \ \rho(x)\Lphi(y)=\Lphi(y R(x)), \ \ \ x,y\in C_u(\G).$$
Since $\lm$ and $\rho$ have commuting ranges, we obtain a canonical representation $\lm\times\rho:C_u(\G)\ten_{\max}C_u(\G)\rightarrow\BLTQ$. Composing with $\om_{\Lphi(1)}$, we obtain a state 
$$\mu_0:=\om_{\Lphi(1)}\circ\lm\times\rho\in(C_u(\G)\hten C_u(\G))^*$$ 
by \cite[Theorem~9.4.1]{ER}. Let $1$ denote the unit in $C_u(\G)^*$, and consider the map
$$C_u(\G)^*\ni \mu\mapsto\mu_0\star(1\ten\mu)\in C_u(\G)^*\ehten C_u(\G)^*.$$
By Lemma~\ref{l:dualBA} this map is weak*-weak* continuous. Its pre-adjoint
$$m:=C_u(\G)\hten C_u(\G)\rightarrow C_u(\G)$$
satisfies
\begin{align*}\la m(u_{ij}^\al\ten u_{kl}^\be),\mu\ra&=\la u_{ij}^\al\ten \mu\star u_{kl}^\be,\mu_0\ra=\sum_{n=1}^{n_\be}\la\mu, u_{nl}^\be\ra\la\mu_0,u_{ij}^\al\ten u_{kn}^\be\ra\\
&=\sum_{n=1}^{n_\be}\la\mu, u_{nl}^\be\ra\la\lm(u_{ij}^\al)\rho(u_{kn}^\be)\Lphi(1),\Lphi(1)\ra\\
&=\sum_{n=1}^{n_\be}\la\mu, u_{nl}^\be\ra\vphi(u_{ij}^\al (u_{nk}^\be)^*)=\frac{\delta_{\al\be}\delta_{jk}}{n_\al}\la\mu,u_{il}^\be\ra\\
&=\la u_{ij}^\al\star u_{kl}^\be,\mu\ra.\end{align*}
Hence, convolution product is bounded in the Haagerup tensor product, making $C_u(\G)$ an operator algebra under convolution.

Now, if $f\in\A\cdot\vphi\subseteq C_u(\G)^*$, then the above calculation implies that $m^*(f)=\mu_0\star(1\ten f)\in\A\cdot\vphi\ten\A\cdot\vphi\subseteq M(\G)\ehten M(\G)$. By projectivity of the Haagerup tensor product, the canonical map $C_u(\G)\hten C_u(\G)\rightarrow C(\G)\hten C(\G)$ is a complete quotient map, so that
$$M(\G)\ehten M(\G)\hookrightarrow C_u(\G)^*\ehten C_u(\G)^*$$
completely isometrically. Since $M(\G)$ is the weak* closure of $\LOQ$ in $C_u(\G)^*$, it follows that $m^*|_{M(\G)}:M(\G)\rightarrow M(\G)\ehten M(\G)$ is weak*-weak* continuous and 
$$(m^*|_{M(\G)})_*:C(\G)\hten C(\G)\rightarrow C(\G)$$
implements the convolution product on $C(\G)$.
\end{proof}

We now show that for a large class of compact Kac algebras, the Haagerup dual of the operator algebra $C(\G)$ is precisely $\ell^\infty(\h{\G})$. In particular, for compact groups $G$, the Haagerup dual of the convolution algebra $C(G)$ is the group von Neumann algebra $VN(G)$. Here, we again see a reflection of quantum group duality through our new notion.

\begin{lem}\label{l:rcfactorization}
Let $\G$ be a compact Kac algebra. The canonical inclusion 
$$\LIQ\ni x\mapsto \Lphi(x)\in\LTQ$$
extends to a complete contraction when $\LTQ$ is furnished with the column, $\LTQ_c$, or row, $\LTQ_r$, operator space structure.
\end{lem}

\begin{proof}
Let $x =[x_{ij}] \in M_n(\LIQ)$.  Then
\begin{align*}\norm{[\Lphi(x_{ij})]}_c&=\norm{[\sum_{k=1}^n\la\Lphi(x_{kj}),\Lphi(x_{ki})\ra]}^{1/2}=\norm{[\sum_{k=1}^n\vphi(x_{ki}^*x_{kj})]}^{1/2}=\norm{\vphi_n([[x_{ij}]^*[x_{ij}])}^{1/2}\\
&\leq\norm{[[x_{ij}]^*[x_{ij}]}^{1/2}=\norm{[x_{ij}]}.\end{align*}
Using the traciality of $\vphi$, we also have
\begin{align*}\norm{[\Lphi(x_{ij})]}_r&=\norm{[\sum_{k=1}^n\la\Lphi(x_{ik}),\Lphi(x_{jk})\ra]}^{1/2}=\norm{[\sum_{k=1}^n\vphi(x_{jk}^*x_{ik})]}^{1/2}\\
&=\norm{[\sum_{k=1}^n\vphi(x_{ik}x_{jk}^*)]}^{1/2}=\norm{\vphi_n([[x_{ij}][x_{ij}]^*)}^{1/2}\\
&\leq\norm{[[x_{ij}][x_{ij}]^*}^{1/2}=\norm{[x_{ij}]}.\end{align*}
\end{proof}

\begin{lem}\label{l:antipodeformula2} Let $\G$ be a compact Kac algebra, $x\in\LIQ$, and $f\in\LOQ$. Then $(f\star x)\cdot\vphi=(x\cdot\vphi)\star f^{o^*}$, and $(x\star f^*)\cdot\vphi=f^o\star(x\cdot\vphi)$.
\end{lem}

\begin{proof} For $y\in\LIQ$, we have $\la(f\star x)\cdot\vphi,y\ra=\vphi(y(\id\ten f)\Gam(x))=f((\vphi\ten\id)(y\ten1)\Gam(x))$. By the antipode relations we have
\begin{align*}
f((\vphi\ten\id)(y\ten1)\Gam(x))&=f(R((\vphi\ten\id)(x^*\ten1)\Gam(y^*))^*)=\overline{f^o((\vphi\ten\id)(x^*\ten1)\Gam(y^*))}\\
&=\overline{\vphi(x^*(\id\ten f^o)\Gam(y^*))}=\vphi((\id\ten f^{o^*})\Gam(y)x)\\
&=\la(x\cdot\vphi)\star f^{o^*},y\ra.
\end{align*}
The second relation follows from a similar calculation.
\end{proof}

\begin{lem}\label{l:matrixentries} Let $\G$ be a compact Kac algebra, and let $f\in\LOQ$. Then for every $\al\in\Irr$,
$1\leq i,j \leq n_\al$,
$$\lm(f)^\al_{ij}=\la f,u_{ij}^{\overline{\al}}\ra$$
\end{lem}

\begin{proof} We first consider the case where $f=x\cdot\vphi$ for $x\in\LIQ$. By Lemma \ref{l:matrixunits}, we have 
\begin{equation}\label{e:coefficients}z_\al\lm(x\cdot\vphi)=n_\al\lm(\vphi^\al\star(x\cdot\vphi))
=\sum_{i,j=1}^{n_\al}\lm(x\cdot\vphi)^\al_{ij}n_\al\lm(\vphi_{ij}^\al).
\end{equation}
Since $\vphi^\al=\vphi^{\al^o}$, Lemma \ref{l:antipodeformula2} implies that
$\vphi^\al\star(x\cdot\vphi)=(x\star \vphi^{\al^*})\cdot\vphi$. But
$x\star\vphi^{\al^*}=(\vphi^{\al^*}\ten\id)\Gam(x)=\sum_{i,j=1}^{n_\al}c^\al_{ij}u_{ij}^\al$ for some set of coefficients
$c_{ij}^\al\in\C$. Applying $\Lphi$, we obtain
$$\sum_{i,j=1}^{n_\al}c^\al_{ij}\Lphi(u_{ij}^\al)=\Lphi((\vphi^{\al^*}\ten\id)\Gam(x))
=(\vphi^{\al^*}\ten\id)(W^*)\Lphi(x)=\lm(\vphi^\al)\Lphi(x).$$
It then follows form the orthogonality relations that
$$c_{ij}^\al=n_\al\la\lm(\vphi^\al)\Lphi(x),\Lphi(u_{ij}^\al)\ra=\la\Lphi(x),\Lphi(u_{ij}^\al)\ra.$$
Using equation (\ref{e:coefficients}) together with the linear independence of $\vphi_{ij}^\al$ in $\LOQ$, we obtain
$$\lm(x\cdot\vphi)^\al_{ij}=\la\Lphi(x),\Lphi(u_{ij}^\al)\ra=\la x\cdot\vphi,u_{ij}^{\overline{\al}}\ra.$$
Then general relation then follows from the density of $\LIQ\cdot\vphi$ in $\LOQ$.
\end{proof}

Note that Lemma \ref{l:matrixentries} implies the following formula for the left fundamental unitary:
\begin{equation}\label{e:W} W=\sum_{\al\in\Irr}\sum_{i,j=1}^{n_\al}u_{ij}^\al\ten e_{ij}^{\overline{\al}},\end{equation}
where the sum converges weak* in $\LIQ\oten\ell^{\infty}(\h{\G})$, and $e_{ij}^\al=n_\al\lm(\vphi_{ij}^\al)$ are the canonical matrix units.

\begin{thm}\label{t:H(C(G))}
Let $\G$ be a compact Kac algebra whose dual $\h{\G}$ is weakly amenable. Then under the convolution product
$$C(\G)^h\cong\ell^\infty(\h{\G}),$$
completely isometrically and weak*-weak* homeomorphically.
\end{thm}

\begin{proof} Fixing the Peter-Weyl basis $(\xi_{ij}^\al):=(\sqrt{n_\al}\Lphi(u^\al_{ij}))$ of $\LTQ$, it follows that 
$$\LTQ_r\hten\LTQ_c\ni \xi_{ij}^\al\ten\xi_{kl}^\be\mapsto \om_{\xi_{kl}^\be,\xi_{ij}^\al}\in\TCQ$$
extends to a completely isometric isomorphism (see \cite[Corollary 5.11]{P}, for instance). Define unitaries $U_r,U_c:\LTQ\rightarrow\LTQ$ by
$$U_r\xi_{ij}^{\al}=\xi_{ij}^{\overline{\al}}, \ \ \ U_c\xi_{ij}^{\al}=\xi_{ji}^{\overline{\al}}, \ \ \ \al\in\Irr, \
1\leq i,j \leq n_\al.$$
By \cite[Theorem 3.4.1, Proposition 3.4.2]{ER}, $U_r$ and $U_c$ define complete contractions on $\LTQ_r$ and $\LTQ_c$, respectively. 
Let $m$ denote the following composition
$$C(\G)\hten C(\G)\xrightarrow{\Lphi\ten\Lphi}\LTQ_r\hten\LTQ_c\xrightarrow{U_r\ten U_c}\LTQ_r\hten\LTQ_c\cong\TCQ\rightarrow\ell^1(\h{\G}),$$
where the final map is given by restriction $\om\mapsto \om|_{\ell^\infty(\h{\G})}$. By Lemma~\ref{l:rcfactorization} and the above discussion, it follows that $m$ is a complete contraction satisfying 
$$m(u_{ij}^\al\ten u_{kl}^\be)=\frac{1}{\sqrt{n_\al n_\be}}\om_{\xi_{lk}^{\overline{\be}},\xi_{ij}^{\overline{\al}}}\bigg|_{\ell^\infty(\h{\G})}.$$
To examine the restricted vector functional, we let $f\in\LOQ$ and calculate
\begin{align*}\frac{1}{\sqrt{n_\al n_\be}}\om_{\xi_{lk}^{\overline{\be}},\xi_{ij}^{\overline{\al}}}(\lm(f)^*)&=\la\lm(f)^*\Lphi(u_{lk}^{\overline{\be}}),\Lphi(u_{ij}^{\overline{\al}})\ra=\la(f^*\ten\id)(W^*)\Lphi(u_{lk}^{\overline{\be}}),\Lphi(u_{ij}^{\overline{\al}})\ra\\
&=\la\Lphi((f^*\ten\id)\Gam(u_{lk}^{\overline{\be}})),\Lphi(u_{ij}^{\overline{\al}})\ra=\sum_{n=1}^{n_\be}\la f^*,u_{ln}^{\overline{\be}}\ra\la\Lphi(u_{nk}^{\overline{\be}}),\Lphi(u_{ij}^{\overline{\al}})\ra\\
&=\frac{\delta_{\al\be}\delta_{jk}}{n_\al}\la f^*,u_{li}^{\overline{\al}}\ra=\frac{\delta_{\al\be}\delta_{jk}}{n_\al}\overline{\la f,u_{li}^{\al}\ra}\\
&=\frac{\delta_{\al\be}\delta_{jk}}{n_\al}\overline{\lm(f)_{li}^{\overline{\al}}}=\frac{\delta_{\al\be}\delta_{jk}}{n_\al}(\lm(f)^*)_{il}^{\overline{\al}}\\
&=\frac{\delta_{\al\be}\delta_{jk}}{n_\al}\om_{il}^{\overline{\al}}(\lm(f)^*),
\end{align*}
where $\om_{il}^\al$ is the evaluation at the $\al$, $i,l$ entry. Since $f\in\LOQ$ was arbitrary, we obtain
$$m(u_{ij}^\al\ten u_{kl}^\be)=\frac{\delta_{\al\be}\delta_{jk}}{n_\al}\om_{il}^{\overline{\al}}.$$
Now, 
$$\LOQ\hten\LOQ\xrightarrow{\lm\ten\lm}c_0(\h{\G})\hten c_0(\h{\G})\xrightarrow{m_{c_0}}c_0(\h{\G})$$
is a complete contraction whose adjoint $\ell^1(\h{\G})\rightarrow\LIQ\ehten\LIQ$ equals $\Gam\circ \lm_*$, where 
$$\lm_*:\ell^1(\h{\G})\ni\hat{f}\mapsto(\id\ten\hat{f})(W)\in C(\G).$$
It follows that $\Gam\circ\lm_*$ maps $\ell^1(\h{\G})$ into $C(\G)\hten C(\G)$. Moreover, given $[\hat{f}_{rs}]\in M_n(\ell^1(\h{\G}))$, equation (\ref{e:W}) entails
$$[\Gam(\lm_*(\hat{f}_{rs}))]=[(\id\ten\hat{f}_{rs})(W_{13}W_{23})]=[\sum_{\al}\sum_{i,j=1}^{n_\al}\la\hat{f}_{rs},e_{il}^{\overline{\al}}\ra u_{ik}^\al\ten u_{kl}^\al],$$
where the sum representing each $r,s$ entry is convergent in the Haagerup norm. Thus,
$$m^{(n)}[\Gam(\lm_*(\hat{f}_{rs}))]=[\sum_{\al}\sum_{i,j=1}^{n_\al}\la\hat{f}_{rs},e_{il}^{\overline{\al}}\ra\om_{il}^{\overline{\al}}]=[\hat{f}_{rs}].$$
Since $\Gam\circ \lm_*$ is a complete contraction it follows that $m$ is a complete quotient map. Since $m$ is also balanced with respect to the $C(\G)$-module structure, it induces a complete quotient map 
$$\tilde{m}:C(\G)\hten_{C(\G)}C(\G)\rightarrow \ell^1(\h{\G}).$$

Now, since $\h{\G}$ is weakly amenable, the algebra $(C(\G),\star)$ has a (two-sided) completely bounded multiplier approximate identity. By Proposition \ref{p:h-module-mb-ai} it follows that $\mathrm{Ker}(m_\star)=N_{C(\G)}$, where $m_\star:C(\G)\hten C(\G)\rightarrow C(\G)$ is the convolution product. Lemma \ref{l:matrixentries} implies
$$\la\lm_*(\om_{il}^{\overline{\al}}),f\ra=\la \om_{il}^{\overline{\al}},\lm(f)\ra=\la u_{il}^{\al},f\ra,$$
so that $\lm_*(\om_{il}^{\overline{\al}})=u_{il}^{\al}$. We therefore have $\lm_*\circ m=m_\star$, i.e., the following diagram commutes

\begin{equation*}
\begin{tikzcd}
C(\G)\hten C(\G)\arrow[rr, "m_{\star}"]\arrow[rd, "m"] &   &C(\G) \\
&\ell^1(\h{\G}) \arrow[ru, "\lm_*"] &
\end{tikzcd}
\end{equation*}

Since $\lm_*$ is injective, it follows that $\mathrm{Ker}(m)=\mathrm{Ker}(m_\star)=N_{C(\G)}$. Hence, $\tilde{m}$ is also injective, and $C(\G)\hten_{C(\G)}C(\G)$ is completely isometrically isomorphic to $\ell^1(\h{\G})$. 

\end{proof}

As an immediate corollary, we can identify the Haagerup duals of the operator algebras $C(G)$ and $C^*_\lm(G)$ under convolution and pointwise product, respectively.

\begin{cor} Let $G$ be a compact group. Then for $C(G)$ under convolution we have
$$C(G)^h=VN(G).$$ 
\end{cor}

\begin{cor} Let $G$ be a weakly amenable discrete group. Then for $C^*_\lm(G)$ under pointwise multiplication we have
$$C^*_\lm(G)^h=\ell^\infty(G).$$
\end{cor}

For a locally compact group $G$, there is a natural $A(G)$-bimodule structure on $C^*_\lm(G)$. One can then consider the Haagerup dual of $C^*_\lm(G)$ in the category of operator $A(G)$-modules. For simplicity we consider the case where $G$ is discrete. A similar argument can be used in the setting of compact Kac algebras.

\begin{prop}\label{p:H-A-C*}
Let $G$ be a weakly amenable discrete group. Then 
$$C^*_\lambda(G)^h= \ell_c^2(G)^*,$$
completely isometrically and weak*-weak* homeomorphically.
\end{prop}

\begin{proof}
First we prove that the pointwise action $A(G) \otimes C^*_\lambda(G)\rightarrow C^*_\lm(G)$ extends to a complete contraction 
\[
m: A(G) \otimes^h C^*_\lambda(G) \rightarrow \ell^2_c(G).
\]
Note that the canonical inclusions $i_1: A(G) \rightarrow C_0(G)$ and $i_2: C_0(G) \rightarrow \mc{K}(\ell^2(G))$ are both complete contractions. Let $\Lphi$ be the map from Lemma~\ref{l:rcfactorization}. Then
\begin{equation}\label{eq:A-C0-K}
A(G) \otimes^h C^*_\lambda(G) \xrightarrow{i_1 \otimes \Lphi} C_0(G) \otimes^h \ell^2_c(G) \xrightarrow{i_2 \otimes \id} \mc{K}(\ell^2(G)) \otimes^h \ell^2_c(G).
\end{equation}
is a complete contraction. By \cite[Proposition~9.3.2]{ER} the dual pairing 
\[
\ell^2_c(G)^*  {\otimes}^h \ell^2_c(G)\cong \ell^2_c(G)^* \pten \ell^2_c(G) \rightarrow \Bbb{C}, \quad \bar{\eta} \otimes \xi \mapsto \langle \xi, \eta\rangle
\]
is completely contractive. By Propositions 9.2.7, 9.3.4,  and 9.3.5 in \cite{ER}, the composition 
\[
\mc{K}(\ell^2(G)) \otimes^h \ell^2_c(G)=(\ell_c^2(G) \otimes^h \ell^2_c(G)^*) \otimes^h \ell^2_c(G) = \ell_c^2(G) \otimes^h (\ell^2_c(G)^* \otimes^h \ell^2_c(G))\rightarrow \ell_c^2(G)
\]
is the canonical action of $\mc{K}(\ell^2(G))$ on $\ell^2(G)$. Composing with \eqref{eq:A-C0-K}, we see that the mapping
\begin{equation*}\label{eq:kappa}
m: A(G) \otimes^h C^*_\lambda(G) \rightarrow \ell^2_c(G), \quad u \otimes x \rightarrow \Lphi(u \cdot x)
\end{equation*}
is a complete contraction, where $u \cdot x$ is the pointwise action of $A(G)$ on $C^*_\lm(G)$.

Now, for each $\xi=(\xi_g)_{g \in G} \in \ell^2(G)$ define 
\[
\psi( \xi) = \sum_{g \in G} \xi_g \delta_g \otimes \lambda_g.
\]
We show that $\psi$ is a completely contractive right inverse to $m$. Let $I_{n}$ denote the $n \times n$ identity matrix and $e_{ij}$ the standard matrix units in $M_{n}$. For $\xi=[\xi_{ij}]\in M_n(\ell^2(G))$ we have
\begin{equation}\label{eq:main-thm-1}
\psi^{(n)}(\xi)= \left[ \sum_{g \in G} \xi_{ij}^g \delta_g \otimes \lambda_g \right]  = u\odot y,\end{equation}
where $u=\left( I_{n} \otimes [\; \cdots \; \delta_g\; \cdots \; ]_{g \in G}  \right)\in M_n(M_{1,\infty}(A(G)))$ and 
$$y=\sum_{i,j} e_{i,j} \otimes
\left[
 \begin{array}{c}
\vdots\\
\\
\xi_{i,j}^g \lambda_g\\
\\
\vdots
\end{array} 
\right]_{g \in G}\in M_n(M_{\infty,1}(C^*_\lm(G)).$$
To justify the inclusions, suppose $x=[x_{ij}]\in M_n(VN(G))$. Then
\[
\norm{ \la\la x, [\; \cdots \; \delta_g\; \cdots \; ]_{g \in G} \ra\ra }_{M_n(M_{1,\infty})}= \norm{ [ \Lphi(x_{ij})]}_{M_n(M_{1,\infty})}
\]
where we view $\Lphi(x_{ij})$ as the infinite row $[\Lphi(x_{ij})^g]_{g\in G}$. Hence,
\begin{align*}
\norm{ [ \Lphi(x_{ij})]}_{M_n(M_{1,\infty})}&=\norm{ [\Lphi(x_{ij})^g] [\Lphi(x_{ij})^g]^*}_{M_n}^{\frac{1}{2}}=
\left\| \left[ \sum_{k=1}^n \sum_{g \in G}  \Lphi(x_{ik})^g \overline{\Lphi(x_{kj})^g}\right] \right\|_{M_n}^{\frac{1}{2}}\\
&=\left\| \left[ \sum_{k=1}^n \langle \Lphi(x_{ik}),  \Lphi(x_{jk}) \rangle \right]  \right\|_{M_n}^{\frac{1}{2}}=\norm{ [\Lphi(x_{ij})]]}_{M_n(\ell^2_r(G))}\leq \norm{[x_{ij}]}_{M_n(VN(G))}.
\end{align*}
It follows that $\norm{u}_{M_n(M_{1,\infty}(A(G)))}\leq 1$.

Next,  
\begin{align*}
\norm{y}_{M_n(M_{\infty, 1}(C^*_\lambda(G)))}&=\norm{y^*y}_{M_n(M_{\infty, 1}(C^*_\lambda(G)))}^{\frac{1}{2}} =\left\| \left[ \sum_{k=1}^m \left[ \cdots \; \overline{\xi_{ki}^g} \lambda_{g^{-1}} \; \cdots \right] \left[
 \begin{array}{c}
\vdots\\
\\
\xi_{kj}^g \lambda_g\\
\\
\vdots
\end{array} 
\right]  \right] \right\|_{M_n(C^*_\lambda(G))}^{\frac{1}{2}}\\
&=\left\| \left[ \sum_{k=1}^m \sum_{g \in G} \xi_{kj}^g \overline{\xi_{ki}^g} \lambda_e \right] \right\|_{M_n(C^*_\lambda(G))}^{\frac{1}{2}}=\left\| \left[ \sum_{k=1}^m \langle \xi_{kj},  \xi_{ki}\rangle  \right] \otimes \lambda_e \right\|_{M_n(C^*_\lambda(G))}^{\frac{1}{2}}\\
&=\left\| \left[ \sum_{k=1}^m \langle \xi_{kj},  \xi_{ki}\rangle  \right] \right\|_{M_n}^{\frac{1}{2}} = \norm{\xi}_{M_n(\ell^2_c(G))}.
\end{align*}
Thus, 
\[
\norm{\psi^{(n)}(\xi)}_{M_{n}(A(G)\otimes^h C^*_\lambda(G))}=\norm{u\odot y}_{M_{n}(A(G)\otimes^h C^*_\lambda(G))}\leq \norm{\xi}_{M_{n}(\ell^2_c(G))}
\]
and therefore, $\psi$ is a complete contraction. By construction it is a right inverse to $m$, implying that $m$ is a complete quotient map. 
  
Now if $G$ is weakly amenable, then $A(G)$ has an approximate identity $(u_i)$ bounded in its cb-multiplier norm. By a similar argument to the proof of Proposition \ref{p:h-module-mb-ai} it follows that $\mathrm{Ker}(m)=N_{C^*_\lm(G)}$ and $\mathrm{Ker}(m_{\ell^2_c(G)})=N_{\ell^2_c(G)}$, where $m_{\ell^2_c(G)}:A(G)\hten\ell^2_c(G)\rightarrow \ell^2_c(G)$, 
$$N_{C^*_\lm(G)}=\la u\cdot v\ten x - u\ten v\cdot x\mid u,v\in A(G), \ x\in C^*_\lm(G)\ra$$
and similarly for $N_{\ell^2_c(G)}$. Let 
\begin{equation}\label{e:eqfinal}(\id\ten\Lphi):A(G)\hten_{A(G)}C^*_\lm(G)\rightarrow A(G)\hten_{A(G)} \ell^2_c(G)\end{equation}
denote the induced morphism. If $X\in A(G)\hten C^*_\lm(G)$ with $(\id\ten \Lphi)(X+N_{C^*_\lm(G)})=N_{\ell^2_c(G)}$, then $\widetilde{m_{\ell^2_c(G)}}((\id\ten\Lphi)(X+N_{C^*_\lm(G)}))=0$. But
$$X+N_{C^*_\lm(G)}=\lim_i(u_i\ten 1)X+N_{C^*_\lm(G)}=\lim_i u_i\ten m_{C^*_\lm(G)}(X)+N_{C^*_\lm(G)},$$
so that
\begin{align*}0&=\widetilde{m_{\ell^2_c(G)}}((\id\ten\Lphi)(X+N_{C^*_\lm(G)}))\\
&=\lim_i\widetilde{m_{\ell^2_c(G)}}((\id\ten\Lphi)(u_i\ten m_{C^*_\lm(G)}(X)+N_{C^*_\lm(G)}))\\
&=\lim_i \widetilde{m_{\ell^2_c(G)}}(u_i\ten\Lphi(m_{C^*_\lm(G)}(X)))\\
&=\lim_i \Lphi(u_i\cdot m_{C^*_\lm(G)}(X))\\
&=\Lphi(m_{C^*_\lm(G)}(X)).
\end{align*}
By injectivity of $\Lphi$, it follows that $X\in\mathrm{Ker}(m_{C^*_\lm(G)})=N_{C^*_\lm(G)}$. Hence, the map (\ref{e:eqfinal}) is injective. Since the diagram below commutes

\begin{equation*}
\begin{tikzcd}
A(G) \otimes^h_{A(G)} C^*_\lambda(G)\arrow[r, "\tilde{m}"]\arrow[d, "(\id\ten\Lphi)"]   &\ell^2_c(G) \\
A(G)\hten_{A(G)} \ell^2_c(G)\arrow[r, "\widetilde{m_{\ell^2_c(G)}}"]&\ell^2_c(G) \arrow[u, equals]
\end{tikzcd}
\end{equation*}
it follows that $\tilde{m}$ is injective. Thus, $A(G)\hten_{A(G)}C^*_\lm(G)\cong\ell^2_c(G)$ completely isometrically, and $C^*_\lambda(G)^h=\ell^2_c(G)^*$. 
\end{proof}

\section*{Acknowledgements}

The first author was partially supported by a Carleton-Fields Institute Postdoctoral Fellowship. The second author was partially supported by the NSERC Discovery Grant 1304873. The third author was partially supported by an NSERC Discovery Grant.

\end{spacing}


\begin{thebibliography}{00}


\bibitem{AC}
M. Alaghmandan and J. Crann, 
\textit{Character density in central subalgebras of compact quantum groups}. 
Canad. Math. Bull. 60 (2017), no. 3, 449-461.
 
\bibitem{ACN}
M. Alaghmandan, J. Crann, and M. Neufang,
\textit{Mapping ideals and multipliers for quantum groups}. 
arXiv:1803.08342.
 

 
\bibitem{B} D. P. Blecher,
\textit{A completely bounded characterization of operator algebras}. 
Math. Ann. 303 (1995), no. 2, 227-239. 
 
\bibitem{BL}
D. P. Blecher, C. Le Merdy,
\textit{Operator algebras and their modules: an operator space approach.}
London Mathematical Society Monographs. New Series, 30. Oxford Science Publications. The Clarendon Press, Oxford University Press, Oxford, 2004. 


\bibitem{BL2}
D. P. Blecher, C. Le Merdy,
\textit{On quotients of function algebras and operator algebra structures on $\ell_p$.}  
J. Operator Theory 34 (1995), no. 2, 315-346. 

\bibitem{BS}
D. P. Blecher, R. R. Smith,
\textit{The dual of the Haagerup tensor product.}
J. London Math. Soc. (2) 45 (1992), no. 1, 126-144.
  


 
\bibitem{CL}
J. Cigler,  V. Losert, P. Michor, 
\textit{Banach modules and functors on categories of Banach spaces}.
Lecture Notes in Pure and Applied Mathematics, 46. Marcel Dekker, Inc., New York, 1979.

\bibitem{C} J. Crann,
\textit{Inner amenability and approximation properties of locally compact quantum groups}.
Indiana Univ. Math. J. to appear. arXiv:1709.01770

\bibitem{D} M. Daws,
\textit{Multipliers, self-induced and dual Banach algebras}.  
Dissertationes Math. 470 (2010), 62 pp. 
 




\bibitem{ER} 
 E. G. Effros, Z.-J. Ruan,
 \textit{Operator spaces}. 
 London Mathematical Society Monographs. New Series, 23. The Clarendon Press, Oxford University Press, New York, 2000. 

 

\bibitem{Gil} J. E. Gilbert,
\textit{$L^p$-convolution operators and tensor products of Banach spaces I,II,III}.
preprints, 1973-1974.

\bibitem{G}
M. Grosser, 
\textit{Bidualr\"{a}ume und Vervollst\"{a}ndigungen von Banachmoduln.}
Lecture Notes in Mathematics, 717. Springer, Berlin, 1979.

\bibitem{K} J. Kustermans,
\e{Locally compact quantum groups in the universal setting}.
Internat. J. Math. {\bf 12}(2001), no.~3, 289--338.
 
 
\bibitem{KV1} J. Kustermans and S. Vaes,
\textit{Locally compact quantum groups}.
Ann. Sci. \'{E}cole Norm. Sup. (4) 33 (2000), no. 6, 837-934.


\bibitem{KV2} J. Kustermans and S. Vaes,
\textit{Locally compact quantum groups in the von Neumann algebraic setting}.
Math. Scand. 92 (2003), no. 1, 68-92.

\bibitem{NT2} S. Neshveyev and L. Tuset,
\e{Compact Quantum Groups and Their Representation Categories}.
Cours Sp\'{e}cialis\'{é}s--Collection SMF, 20, 2014.

\bibitem{P} G. Pisier, 
\textit{Introduction to operator space theory}.
London Mathematical Society Lecture Note Series, 294. Cambridge University Press, Cambridge, 2003.

\bibitem{Ri} M. A. Rieffel, 
\textit{Induced Banach representations of Banach algebras and locally compact groups}.
J. Funct. Anal. 1 (1967), 443-491.
 
 
 




\bibitem{V} N. Th. Varopoulos,
\textit{A theorem on operator algebras}. 
Math. Scand. 37 (1975), no. 1, 173-182.


\end{thebibliography}
\end{document}